\documentclass{article}

\usepackage[square,numbers]{natbib}
\usepackage[letterpaper]{geometry}
\usepackage{amsmath,amsfonts,amssymb,mathrsfs,bbm,amsthm}
\usepackage{enumerate}
\usepackage{hyperref}

\usepackage{mathtools}
\usepackage{subfigure}
\usepackage[charter,cal=cmcal]{mathdesign}  % for mathcal fonts

\usepackage[usenames,svgnames,dvipsnames]{xcolor}

\usepackage{ulem}
\usepackage{wrapfig}

\newcommand{\mtrx}[1]{\boldsymbol{#1}}
\newcommand{\vctr}[1]{\boldsymbol{#1}}
\newcommand{\lattice}[1]{\mathcal{L}_{#1}}
\newcommand{\vzero}{\vctr{0}}
\newcommand{\vone}{\vctr{1}}
\newcommand{\vunit}[1]{\vctr{e}_{#1}}
\newcommand{\fcc}{\mathrm{fcc}}
\newcommand{\bcc}{\mathrm{bcc}}
\newcommand{\mXi}{\mtrx{\Xi}}
\newcommand{\mZeta}{\mtrx{\mathrm{Z}}}
\newcommand{\matid}[1]{\mtrx{I}_{#1}}
\newcommand{\matone}[1]{\mtrx{J}_{#1}}
\newcommand{\mG}{\mtrx{G}}

\newcommand{\mQ}{\mtrx{Q}}
\newcommand{\mA}{\mtrx{A}}
\newcommand{\mGamma}{\mtrx{\mathit{\Gamma}}}
\newcommand{\Mproj}[1]{\mtrx{P}_{#1}}
\DeclareMathOperator{\rank}{rank}

\DeclareMathOperator{\vol}{vol}
\newcommand{\vx}{\vctr{x}}
\newcommand{\va}{\vctr{a}}
\newcommand{\vy}{\vctr{y}}
\newcommand{\vj}{\vctr{j}}
\newcommand{\vp}{\vctr{p}}
\newcommand{\vq}{\vctr{q}}
\newcommand{\vu}{\vctr{u}}
\newcommand{\vv}{\vctr{v}}
\newcommand{\vxi}{\vctr{\xi}}
\newcommand{\vzeta}{\vctr{\zeta}}

\newcommand{\reffig}[1]{Figure~\ref{#1}}
\newcommand{\reflem}[1]{Lemma~\ref{#1}}
\newcommand{\refcor}[1]{Corollary~\ref{#1}}
\newcommand{\Z}{\mathbb{Z}}
\newcommand{\R}{\mathbb{R}}
\newcommand{\Gfcc}{\mtrx{G}_{\fcc}}
\newcommand{\Gbcc}{\mtrx{G}_{\bcc}}
\newcommand{\perm}{\pi}

\newcommand{\Mperm}{\mtrx{\mathit{\Pi}}}
\newcommand{\Mshift}[2]{\mtrx{\mathit{\Gamma}}_{#1,#2}}
\newcommand{\Mswitch}[2]{\mtrx{\mathit{\Sigma}}_{#1,#2}}
\newcommand{\trans}[1]{#1^T}
\newtheorem{lemma}{Lemma}
\newtheorem{corollary}{Corollary}
\newtheorem{definition}{Definition}
\newcommand{\setJ}{\mathcal{J}}
\newcommand{\setV}{\mathcal{V}}
\newcommand{\genmat}{\mtrx{G}}
\newcommand{\An}[1]{\mathcal{A}_{#1}}
\newcommand{\Andual}[1]{\mathcal{A}_{#1}^*}
\newcommand{\Dn}[1]{\mathcal{D}_{#1}}

\newcommand{\WeylG}[1]{W\!\!\left(#1\right)}
\newcommand{\aWeylG}[1]{W_{\!a}\!\!\left(#1\right)}
\newcommand{\WeylGA}[1]{\WeylG{\sysA{#1}}}
\newcommand{\aWeylGA}[1]{\aWeylG{\sysA{#1}}}
\newcommand{\GAn}[1]{\mG_{\!\!\mathcal{A}}}
\newcommand{\GAndual}[1]{\mG_{\!\!\mathcal{A}^*}}
\newcommand{\Gi}{G_1}
\newcommand{\Gii}{G_2}
\newcommand{\Giii}{G_3}
\newcommand{\Gsym}[1]{S_{#1}}
\newcommand{\Gshift}[1]{C_{#1}}
\newcommand{\simplex}{\sigma}
\newcommand{\CH}[1]{\mathcal{CH}\left(#1\right)}
\newcommand{\ZAn}[1]{\Z_{\!\!\mathcal{A}}^{#1}}
\newcommand{\ZDn}[1]{\Z_{\!\!\mathcal{D}}^{#1}}
\newcommand{\ZAndual}[1]{\Z_{\!\!\mathcal{A}^*}^{#1}}
\newcommand{\pln}[1]{\mathbb{H}_{#1}}
\newcommand{\plnone}{\pln{\vone}}
\newcommand{\vor}[2]{V_{#1}\left(#2\right)}
\newcommand{\vorzero}{\vor{}{\vzero}}
\newcommand{\simplexCube}{\tau}
\newcommand{\rhomb}[1]{{\Diamond}_{#1}}
\newcommand{\cube}[2]{\square^{#1}_{#2}}
\newcommand{\sysA}[1]{\mathscr{A}_{#1}}
\begin{document}

\title{Algebraic Characterization of the Voronoi Cell Structure of the $\An{n}$ Lattice}
\author{Minho Kim}

\maketitle

\begin{abstract}
We 
characterized the combinatorial 
structure of the Voronoi cell
of the $\An{n}$ lattice in arbitrary dimensions.
Based on the well-known fact that the Voronoi cell is 
the disjoint union of $(n+1)!$ congruent simplices,
we show that it is the disjoint union of
$(n+1)$ congruent hyper-rhombi, which are 
the generalized rhombi or trigonal trapezohedra.
The explicit structure of the faces
is investigated,
including the fact that all the $k$-dimensional faces, $2\le k\le n-1$,
are hyper-rhombi.
We show it to be the vertex-first orthogonal projection
of the $(n+1)$-dimensional unit cube. Hence
the Voronoi cell is a zonotope.
We prove that in low dimensions ($n\le 3$)
the Voronoi cell can be understood as 
the section of that of the $\Dn{n+1}$ lattice
with the hyperplane orthogonal to the diagonal direction.
We provide all the explicit coordinates and transformation matrices
associated with our analysis.
Most of our analysis
is algebraic and
easily accessible to those less familiar with 
the Coxeter-Dynkin diagrams.
\end{abstract}

\section{Introduction}
Due to their
higher symmetry compared to other lattices,
the root lattices have been
investigated in various fields 
including 
crystallography, 
signal sampling,
data discretization, etc.
In many applications, the Voronoi cells
of root lattices play a crucial
role.

\citet{conway82voronoi}
derived the Voronoi cells of root lattices
(and their dual lattices) based on the Coxeter-Dynkin diagrams.
Specifically, they showed that the Voronoi cell
$\vorzero$
for a root lattice $\lattice{}$ is the 
    union of the images of its fundamental simplex
    under its Weyl group $\aWeylG{\lattice{}}$.
\citet{moody92voronoi}
determined the number of facets of all dimensions of
the Voronoi cells of the root lattices based on the
the decorated Coxeter-Dynkin diagrams.
\citet{koca18explicit}
analyzed the structure of the $(n-1)$-dimensional faces of the
Voronoi cells and related polytopes of the $\An{n}$ and $\Dn{n}$
lattices and their duals.
They showed that the $(n-1)$-dimensional faces of the Voronoi cells of the $\An{n}$
lattice are the generalized rhombohedra, which we call \emph{hyper-rhombi} in this work.

\section{Notations}
Our characterization uses the following notations.
\begin{itemize}
\item
$\vone$ and $\vzero$ are the (column) vectors composed of $1$s and $0$s
whose dimensions are determined in the context.
\item
$\vunit{i}$ is the $i$-th standard unit (column) vector.
\item
we denote the vectors with repeating elements
as below  \cite{conway82voronoi}.:
    \[
        (
            \underbrace{a,\dots,a}_{\alpha},
            \underbrace{b,\dots,b}_{\beta}
        )=:
    (a^\alpha,b^\beta).
    \]
\item
    $\matid{n}$ is the identity and $\matone{n}:=\vone\trans{\vone}$.
\item
    $\Mswitch{i}{j}$ is the \emph{row exchange}
    matrix of the $i$-th and $j$-th rows.
\item
    The $n\times n$ \emph{circular shift} matrix
    is defined as
\begin{equation}
\label{eq:M.shift}
        \Mshift{n}{j}:=
            \trans{
            \begin{bmatrix}
                \vunit{n-j+1}   &
                \cdots  &
                \vunit{n}   &
                \vunit{1}   &
                \cdots  &
                \vunit{n-j}
            \end{bmatrix}
            }
\end{equation}
which is obtained by circularly shifting the rows of $\matid{n}$ downward by $j$.

\item
The $n$-dimensional (closed) unit cube embedded in $\R^n$
is defined as
\[
    \cube{n}{}:=
    [0,1]^n
    =
    \left\{
        \sum_{i=1}^nt_i\vunit{i}:\vunit{i}\in\R^n,0\le t_i\le 1            
    \right\}
\]
and
\[
    \cube{n}{\setJ}:=
    \left\{
        \sum_{i\in\setJ}
        t_i\vunit{i}:
        \vunit{i}\in\R^n,
        0\le t_i\le 1
    \right\},
    \quad
    \setJ\subset\{1,\dots,n\}
\]
is the $(n-\#\setJ)$-dimensional face 
of $\cube{n}{}$
adjacent to $\vzero$
and
spanned by the unit vectors
excluding those associated with the indices in $\setJ$.
\item
The \emph{convex hull} of the point set $\{\vp_1,\dots,\vp_k\}$ is
denoted as
\[
    \CH{\{\vp_1,\dots,\vp_k\}}.
\]
\end{itemize}
\section{Background}
In this section we review the basic knowledge about the point groups
(represented by matrix groups),
the root system $\sysA{n}$,
and lattices.

\subsection{Point Groups}
A \emph{point group}
is a group composed of
\emph{isometries}  that have a fixed point.
Here we represent a point group
by a  \emph{matrix group}
with $\vzero$ as the fixed point.
Below are some point groups 
to be used in the main discussion.

The \emph{symmetric group}
$\Gsym{n}$ 
is the finite group composed of $n!$ permutation matrices.
Below is a well-known triangulation of the unit cube 
based on the symmetric group.
\begin{figure*}
\begin{center}
\scalebox{1}{
\includegraphics[width=.3\textwidth]{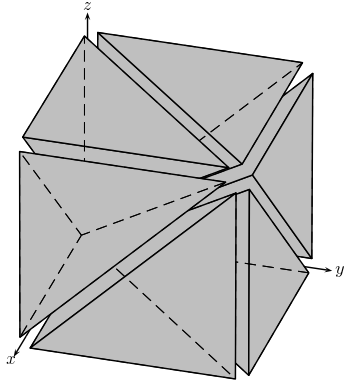}
}
\end{center}
\caption[]{
    Kuhn triangulation of the $3$-dimensional unit cube.
}
\label{fig:kuhn}
\end{figure*}
\begin{lemma}
\label{lem:kuhn}
Given the $n$-dimensional simplex
\begin{equation}
\label{eq:kuhn.simplex}
    \simplexCube_0
    :=
    \CH{\{\vu_j:0\le j\le n\}}
    \text{ where }
    \vu_j:=
        \vzero+\sum_{i=1}^j\vunit{i}=
        (1^j,0^{n-j}),
\end{equation}
the (non-trivial) disjoint union of the $n!$ images of 
$\simplexCube_0$
under $\Gsym{n}$ forms the unit cube
$\cube{n}{}$, i.e.,
\[
    \cube{n}{}=
    \bigsqcup_{
        \mathclap {
            \Mperm \in \Gsym{n}
            }
            }
    \Mperm\tau_0.
\]
In addition, they all share the edge connecting $\vzero$ and $\vone$.
Refer to \reffig{fig:kuhn} for $n=3$.
\end{lemma}
\begin{proof}
Refer to
        \citet{kuhn60some}.
This decomposition is called
the 
\emph{Kuhn triangulation}.
Since
$\vzero$ and $\vone$
are invariant under $\Gsym{n}$,
the edge connecting them is shared by all the simplices.
\end{proof}

\begin{corollary}
\label{cor:kuhn}
Let
\[
    \mA
    =
    \alpha\matid{n}
    +\beta\matone{n},
    \quad
    \alpha,\beta\in\R
   \] be invertible
and $\va_j$ be the $j$-th column of $\mA$.
The disjoint union of the $n!$ images
of the $n$-simplex
\[
    \CH{
    \left\{
        \vzero+
        \sum_{i=1}^j\va_i:
        0\le j\le n
    \right\}
    }
\]
under $\Gsym{n}$ forms an $n$-dimensional parallelepiped
$\mA\cube{n}{}$
and the simplices share the edge
connecting
$\vzero$ and $\sum_{j=1}^n\va_j$.
\end{corollary}

\begin{proof}
    The claim holds since, for $\Mperm\in\Gsym{n}$, 
    \[
               \Mperm\mA
               =
               \mA\Mperm
    \]
    and therefore
    \[
            \Mperm
            \left(
                \sum_{j\in\setJ}\va_j
            \right)
        =
            \Mperm
            \left(
                \sum_{j\in\setJ}\mA\vunit{j}
            \right)
        =
            \mA
            \left(
            \Mperm
                \sum_{j\in\setJ}\vunit{j}
            \right),
            \quad
    \setJ\subset \{1,\dots,n\}.
    \]
\end{proof}

The \emph{circular shift group} 
\[
    \Gshift{n}:=
        \{
            \Mshift{n}{i}
        :0\le i\le n-1\},
\]
    is the finite group composed of $n$ circular shift matrices
\eqref{eq:M.shift}.
Given a set of the hyperplanes through $\vzero$,
A \emph{reflection group} is 
a discrete group
generated by the reflections
w.r.t. the given hyperplanes through $\vzero$
\cite{coxeter34discrete}.

\subsection{Root System $\sysA{n}$}
The \emph{root system} $\sysA{n}$ 
is
composed of the following $n(n+1)$ \emph{roots} 
\[
    \sysA{n}=
    \{
            \vunit{i}-\vunit{j}\in\R^{n+1}:
            1\le i\ne j\le n+1
    \}
\]
embedded in the $n$-dimensional hyperplane 
\begin{equation}
\label{eq:plane.one}
    \plnone:=
    \left\{\vx\in\R^{n+1}:\vx\cdot\vone=0
    \right\}.
\end{equation}
The reflections w.r.t. the hyperplanes orthogonal
to the roots in $\sysA{n}$
form a finite reflection group called the \emph{Weyl group} $\WeylGA{n}$.

\subsection{Lattices}
Given an $m\times n$ generator matrix $\genmat$ with $m\ge n$ and $\rank(\genmat)=n$,
 the \emph{lattice} $\lattice{n}$ embedded in $\R^m$ is defined as all the integer linear combinations
 of the columns of $\genmat$:
\[
    \lattice{n}:= \genmat \Z^n=\{\genmat\vj\in\R^m:\vj\in\Z^n\}.
\]
The \emph{dual lattice} of $\lattice{n}$ is defined as
\[
    \lattice{n}^* := \{\vx\in\R^m:\vx\cdot\vy\in\Z,\forall\vy\in\lattice{n}\}.
\]
The \emph{Voronoi cell}
at $\vp\in\lattice{n}$
is defined as the set of points closest to $\vp$ among all the lattice points of $\lattice{n}$
\[
    \vor{\lattice{n}}{\vp}:=
    \{
        \vx\in\R^m:
        \|\vx-\vp\|\le\|\vx-\vq\|,
        \text{ for all }\vq\in\lattice{n}
    \}.
\]

\section{The Root Lattice $\An{n}$}

\begin{figure*}[ht]
\def\scale{.55}
\begin{center}
\subfigure[$\WeylGA{2}$]{
    \label{fig:W2}
    \includegraphics[width=.31\textwidth]{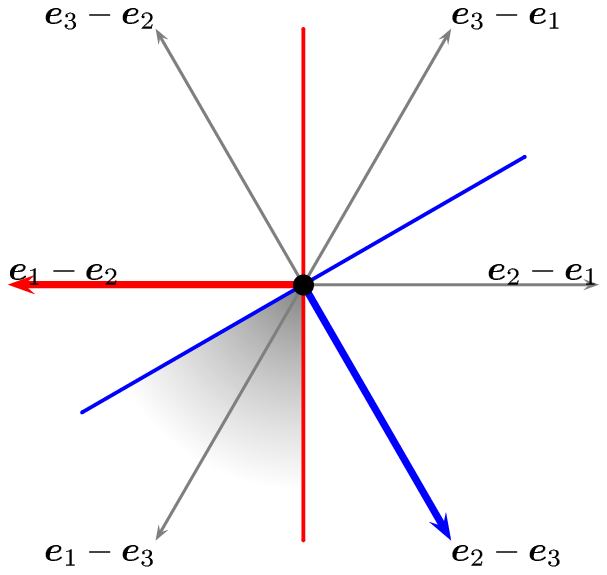}
}
\subfigure[$\aWeylGA{2}$]{
    \label{fig:Wa2}
    \includegraphics[width=.31\textwidth]{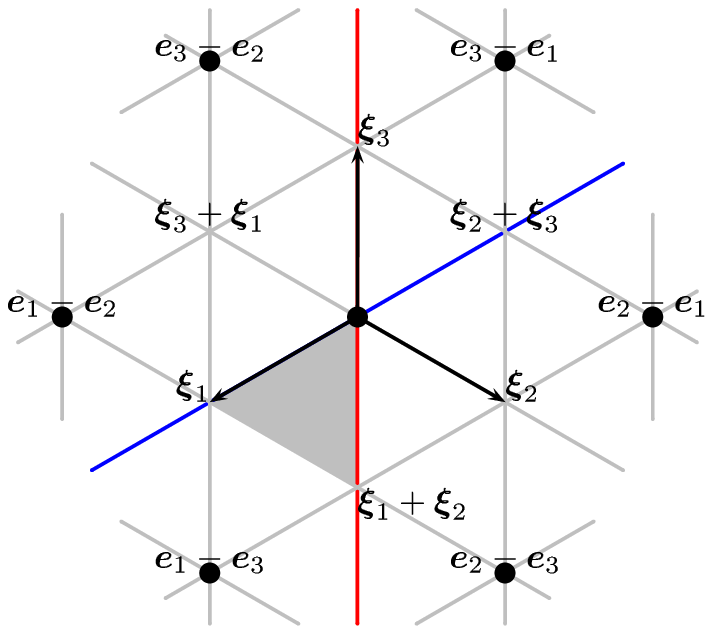}
}
\subfigure[$\vorzero$]{
    \label{fig:A2.V}
    \includegraphics[width=.31\textwidth]{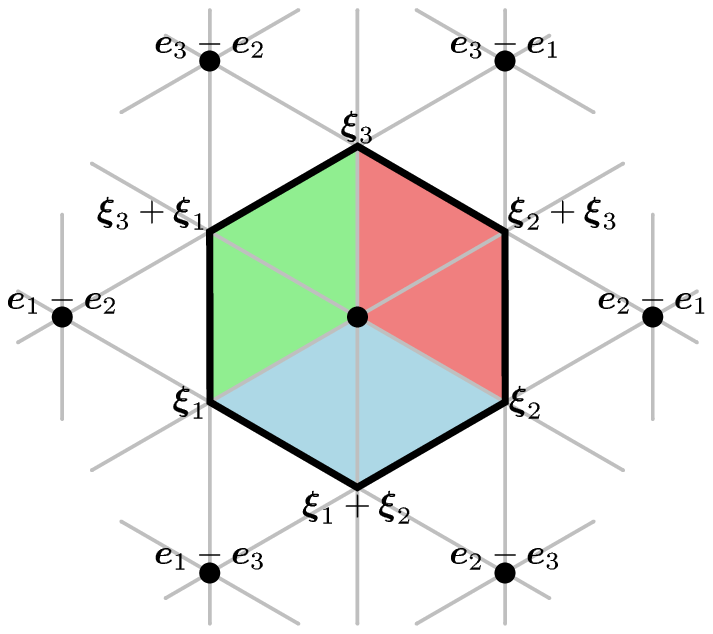}
}
\end{center}
\caption[]{
    \subref{fig:W2}
    $\WeylGA{2}$
    with the fundamental region (gray area),
    \subref{fig:Wa2}
    $\aWeylGA{2}$
    with the fundamental simplex (gray triangle),
    and
    \subref{fig:A2.V}
    $\vorzero$ 
    (thick black hexagon)
    decomposed into three rhombi,
    $\rhomb{1}$ (red),
    $\rhomb{2}$ (green),
    and
    $\rhomb{3}$ (blue).
}
\end{figure*}

\begin{figure*}[ht]
\begin{center}
\subfigure[$\WeylGA{3}$]{
    \label{fig:W3}
    \includegraphics[height=.3\textwidth]{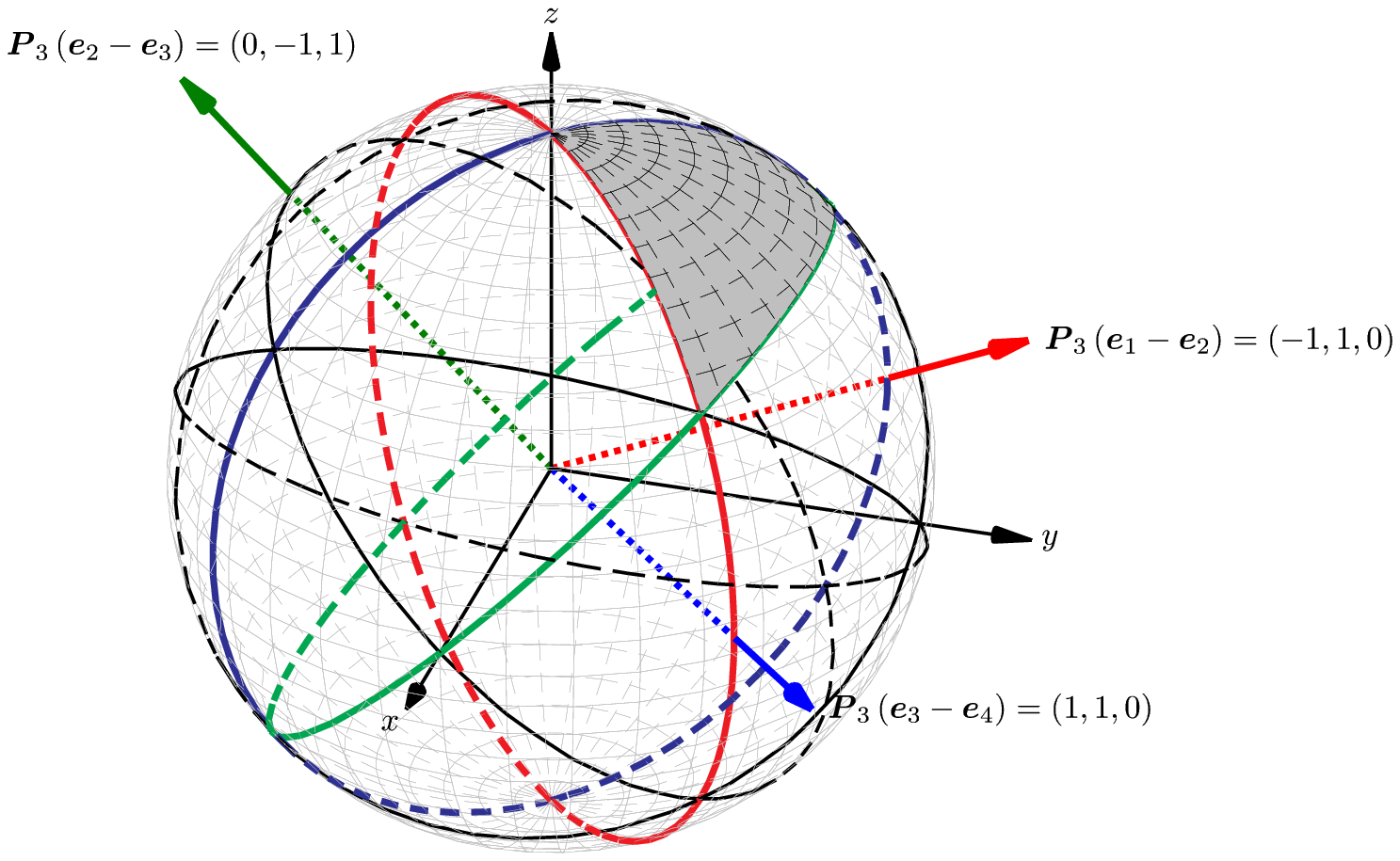}
}
\subfigure[$\aWeylGA{3}$]{
    \label{fig:Wa3}
    \includegraphics[height=.3\textwidth]{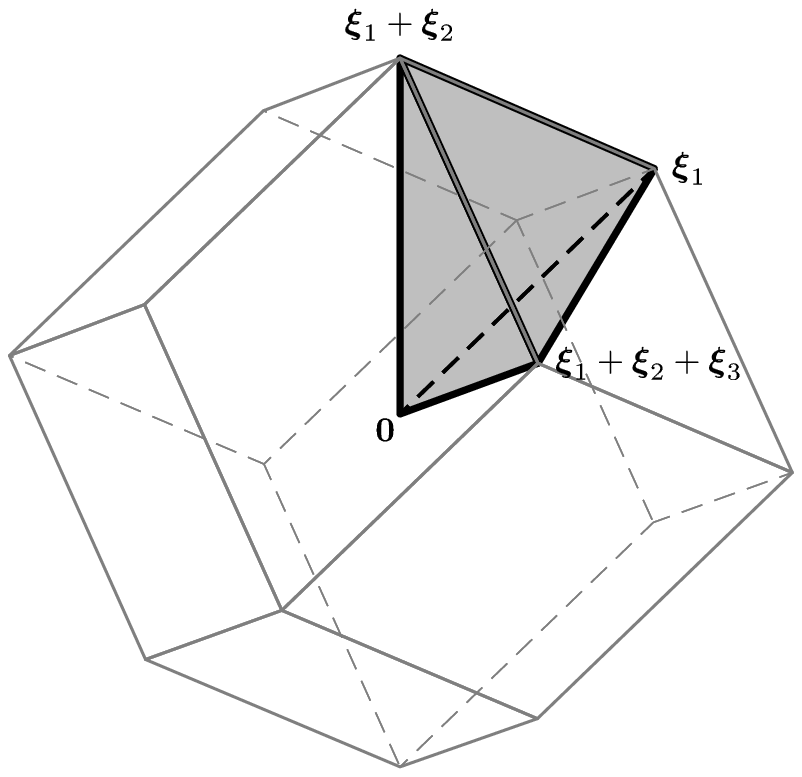}
}
\end{center}
\caption[]{
        \subref{fig:W3}
        The six hyperplanes (denoted by the great circles) associated with 
        the twelve roots of $\sysA{3}$,
        projected onto $\R^3$ by $\Mproj{3}$.
        The three colored hyperplanes
        form the fundamental region whose intersection with the
        sphere is colored in gray.
        \subref{fig:Wa3}
        The Voronoi cell (rhombic dodecahedron) of the $\An{3}$ lattice with
        the fundamental simplex of $\sysA{3}$ (gray tetrahedron).
}
\end{figure*}

In this section, we review the basic properties of the $\An{n}$ lattice
based on an explicit representation given by a generator matrix.

\subsection{Definition and Basic Properties}
The $\An{n}$ lattice  
is composed of all the integer linear combinations
of the roots of $\sysA{n}$
and usually defined 
as
\[
    \ZAn{n}:=
    \GAn{n}\Z^n
    =
    \left\{
            \vx\in\Z^{n+1}:
            \vx\cdot\vone=0
    \right\}
    =
    \Z^{n+1}\cap\plnone
\]
embedded in 
$\plnone$
\eqref{eq:plane.one}
and generated by
the $(n+1)\times n$ generator matrix 
\cite{conway98sphere}
\[
    \GAn{n}:=
    \begin{bmatrix*}[r]
        1   \\
        -1  &   1   \\
            &   -1  &   \ddots   \\
        &   &   \ddots  &   1 \\
                & &   &   -1&   1        \\
                & &   &   &   -1     \\
    \end{bmatrix*}.
\]
The $n$ 
hyperplanes
orthogonal to them
enclose the \emph{fundamental chamber}
or \emph{fundamental region} (\reffig{fig:W2} and \reffig{fig:W3})
and its images 
under the \emph{Weyl group} $\WeylGA{n}$
fill 
$\plnone$.

Let
\[
    \mXi:=
        \frac{1}{n+1}
    \begin{bmatrix*}[r]
         n   &   -1  &      \dots   &   -1 & -1 \\
        -1   &    n  &      \dots   &   -1 & -1 \\
        \vdots & \vdots & \ddots   & \vdots &\vdots\\
        -1   &   -1  &      \dots   &    n  & -1 \\
        -1   &   -1  &      \dots   &    -1  & n\\
    \end{bmatrix*}
    =\matid{n+1}-\frac{1}{n+1}\matone{n+1} %\matone
\]
be the orthogonal projection matrix
onto $\plnone$
along $\vone$
and 
$\vxi_k$ be the $k$-th column of $\mXi$.
The dual lattice 
of $\ZAn{n}$
is defined as
$\ZAndual{n}:=\GAndual{n}\Z^n$
where 
\cite{conway98sphere}
\[
\GAndual{n}:=
\begin{bmatrix}
    \vxi_1  &   \cdots  &   \vxi_n
\end{bmatrix}.
\]

We can construct both $\An{n}$ and $\Andual{n}$ lattices in $\R^n$  %(not $\plnone$)
using any $n\times (n+1)$ orthogonal transformation that maps $\plnone$ to $\R^n$
and
one of such matrices is
\cite{kim11symmetric}
\[
    \Mproj{n}:=
    \left(
        \matid{n} - \frac{1}{n}
            \left(
                1+\frac{1}{\sqrt{n+1}}
            \right)\matone{n}   
    \right)
    \begin{bmatrix}
        -\matid{n}  &   \vone
    \end{bmatrix}.
\]
For example, 
with
\[
    \Mproj{3}=\frac{1}{2}
        \begin{bmatrix*}[r]
            -1 &  1 &  1 & -1  \\
             1 & -1 &  1 & -1  \\
             1 &  1 & -1 & -1
        \end{bmatrix*}
\]
we get
the following square generator matrices for 
the $\An{3}$ and $\Andual{3}$ lattices
which are equivalent to the FCC (Face-Centered Cubic)
and BCC (Body-Centered Cubic) lattices, respectively.:
\[
        \Mproj{3}\GAn{3}
        =
        \begin{bmatrix*}[r]
            -1 &  0 &  1  \\
             1 & -1 &  1  \\
             0 &  1 &  0
        \end{bmatrix*}
        \text{ and }
        \Mproj{3}\GAndual{3}
        =
        \frac{1}{2}
        \begin{bmatrix*}[r]
            -1 &  1 &  1  \\
             1 & -1 &  1  \\
             1 &  1 &  -1
        \end{bmatrix*}.
\]
Note that
\[
    \left(\Mproj{3}\GAn{3}\right)\Z^3=\Gfcc\Z^3
    \text{ and }
    \Mproj{3}\GAndual{3}=\Gbcc
\]
where
\[
    \Gfcc:=
    \begin{bmatrix}
        0   &   1   &   1   \\
        1   &   0   &   1   \\
        1   &   1   &   0   \\
    \end{bmatrix}
    \text{ and }
    \Gbcc:=
    \frac{1}{2}
    \begin{bmatrix*}[r]
        -1   &   1   &   1   \\
        1   &   -1   &   1   \\
        1   &   1   &   -1   \\
    \end{bmatrix*}
\]
are the well-known symmetric generator matrices for 
the FCC and BCC lattices,
respectively.

We review a well-known fact of the $\An{n}$ lattice.
\begin{lemma}
The reflection group
generated by the reflections 
w.r.t. the hyperplanes orthogonal to the  roots of $\sysA{n}$
is 
isomorphic to
the symmetric group $\Gsym{n+1}$.
\end{lemma}
\begin{proof}
Let
\begin{align*}
    \vy:=
    \vx-2\frac{\vx\cdot(\vunit{i}-\vunit{j})}{\|\vunit{i}-\vunit{j}\|^2}(\vunit{i}-\vunit{j})
    =
    \vx-(\vx(i)-\vx(j))(\vunit{i}-\vunit{j})
\end{align*}
be the reflected image of $\vx$ w.r.t. the hyperplane orthogonal to $\vunit{i}-\vunit{j}$.
Then
\[
    \vy(k)=
    \begin{cases}
        \vx(i)-\vx(i)+\vx(j)=\vx(j)  &   \text{if }k=i \\
        \vx(j)+\vx(i)-\vx(j)=\vx(i)  &   \text{if }k=j \\
        \vx(k)  &   \text{otherwise}.
    \end{cases}
\]
and therefore $\vy=\Mswitch{i}{j}\vx$.
The claim holds since $\Gsym{n+1}$
can be generated by row-exchange matrices.
\end{proof}

\section{Decomposition of the Voronoi Cell into Congruent Simplices}

In this and the following sections we algebraically analyze the detailed 
structure of the Voronoi cell of the $\An{n}$ lattice.
Since all the Voronoi cells are congruent, we only consider
$\vorzero$ at the origin.

If we add the hyperplane
$(\vunit{1}-\vunit{n+1})\cdot\vx=1$
to the hyperplanes (orthogonal to the roots) of $\sysA{n}$,
we obtain an infinite reflection group
called the \emph{affine Weyl group} $\aWeylGA{n}$
    (\reffig{fig:W2}).
The simplex built by `capping' the fundamental region
of $\sysA{n}$ with the additional hyperplane is called
the \emph{fundamental simplex} (or \emph{alcove}) of $\sysA{n}$
\cite{conway98sphere}.

We start with
the following lemma 
by \citet{conway82voronoi}
that shows $\vorzero$ is composed of 
$(n+1)!$ copies of the fundamental simplex
of $\aWeylGA{n}$.
\begin{lemma}
\label{lem:V.decomp.simplex}
    Let
    (Refer to \reflem{lem:kuhn} for $\vu_j$.)
    \begin{equation}
    \label{eq:verts.simplex.0}
    \setV_0:=\left\{
                \vv_j:=
                    \mXi\vu_j
                =
                \vzero+
            \sum_{i=1}^j \vxi_i
            =
        \left(
            \left(
                \frac{n+1-j}{n+1}
            \right)^{j},
            \left(
                -
                \frac{j}{n+1}
            \right)^{n+1-j}
        \right):
            0\le j\le n\right\}.
    \end{equation}
    $\vorzero$
    is the 
    union of the $(n+1)!$ images 
    of the
    fundamental simplex
    \begin{equation}
    \label{eq:simplex.0}
    \simplex_0
    :=
    \CH{\setV_0}
    \end{equation}
    under the symmetric group $\Gsym{n+1}$.
    In other words,
    \[
        \vorzero
        =
        \bigsqcup_{\mathclap{\Mperm\in\Gsym{n+1}}}
        \Mperm\simplex_0.
    \]
    The simplicial face
    $
        \CH{
    \{\vv_1,\dots,\vv_n\}
        }
        $
    is
    called the \emph{roof} of 
        $\simplex_0$. 
    
    Refer to
        \reffig{fig:Wa2} and \reffig{fig:Wa3}
        for $n=2$ and $n=3$, respectively.
\end{lemma}
\begin{proof}
Refer to 
\citet[III.B]{conway82voronoi}.
\end{proof}

Therefore, 
the vertices of $\vorzero$ is the union
of $(n+1)!$ images of $\setV_0$ excluding $\vzero$, i.e.,
    \[
        \bigcup_{i=1}^{n}
        \left\{
                \sum_{ \vzeta\in\mZeta}
                \vzeta:
                \mZeta\subset \mXi,
                \#\mZeta=i
        \right\}
    \]
    and the number of vertices of $\vorzero$ is
\cite{conway91cell}
    \[
        \sum_{i=1}^{n}
        \binom{n+1}{i}
        =
        2^{n+1}-2.
    \]
For example, 
(Note that $\sum_{i=1}^{n+1}\vxi_i=\vzero$.)
for $n=2$,
the six
vertices are
(\reffig{fig:A2.V})
\[
    \left\{
            \pm\vxi_1,
            \pm\vxi_2,
            \pm\vxi_3
    \right\}
\]
and for $n=3$,
the $14$ vertices 
are
    (\reffig{fig:Wa3})
\[
    \left\{
            \pm\vxi_1,
            \pm\vxi_2,
            \pm\vxi_3,
            \pm\vxi_4,
            \pm(\vxi_1+\vxi_2),
            \pm(\vxi_1+\vxi_3),
            \pm(\vxi_1+\vxi_4)
    \right\}.
\]

\section{Decomposition of the Voronoi Cell  into Congruent Hyper-Rhombi}
\label{sec:V.decomp.rhombi}
In this section we show that $\vorzero$
is composed of $(n+1)$ congruent `hyper-rhombi',
a generalization of rhombi
or
\emph{trigonal trapezohedra}
(\emph{rhombohedra} with six congruent rhombic faces).
\begin{definition}
An $n$-dimensional `hyper-rhombus' is defined
as a parallelepiped all of whose faces 
are congruent $(n-1)$-dimensional hyper-rhombi.
This definition is recursive with its base case of $2$-dimensional rhombi
(\reffig{fig:A3.V.rhomb}).
\end{definition}

Let
\[
    \mXi_{\setJ}
    :=
    \mXi\backslash
    \{\vxi_j:j\in\setJ\}
\]
be the set of column vectors of $\mXi$ 
excluding those associated with the indices in $\setJ$.
To avoid notational clutter, we denote
$
    \mXi_{\{i\}}=\mXi_i$ and
    $
    \mXi_{\{i,j\}}=\mXi_{ij}$,
    etc.
Then,
for $\setJ\ne \emptyset$,
\[
    \rhomb{\setJ}:=
    \mXi_{\setJ}[0,1]^{n+1-\#\setJ}
    =
    \left\{\sum_{\mathclap{\substack{i=1\\i\notin\setJ}}}^{n+1}
    t_i
    \vxi_i 
    :
    0\le t_i\le 1\right\}
\]
is
the 
$(n+1-\#\setJ)$-dimensional
parallelepiped spanned by
the directions in $\mXi_{\setJ}$ 
since any $n$ directions of $\mXi$ are linearly independent.

\begin{lemma}
\label{lem:rhombi}
The $n$-dimensional  parallelepipeds
$\left\{\rhomb{j}:1\le j\le n+1\right\}$
are hyper-rhombi.
\end{lemma}
\begin{proof}
Note that the $n$ faces of $\rhomb{j}$ adjacent to $\vzero$
are
\[
    \left\{
    \rhomb{ij}
    : 1\le i\ne j\le n+1
    \right\}.
\]
They are all congruent
with $\rhomb{1,2}$
since
(with $\mXi_{ij}$ viewed as a multi-set)
\[
    \mXi_{ij}=
    \Mswitch{1}{i}
    \Mswitch{2}{j}
    \mXi_{1,2}
\]
and therefore are all congruent.
Similarly, all the  $(n-2)$-dimensional faces of 
$\rhomb{ij}$ adjacent to $\vzero$,
\[
    \left\{\rhomb{ijk}:1\le k\ne i,j\le n+1\right\},
\]
are also congruent parallelepipeds.
Finally,
all
the $2$-dimensional faces spanned by
two columns of $\mXi$ are rhombi since
\[
    \vxi_\alpha\cdot\vxi_\beta
    =
    \begin{cases}
        n/(n+1) &   \alpha=\beta \\
        -1/(n+1)    &   \alpha\ne \beta
    \end{cases}
\]
and
therefore the claim holds.
\end{proof}
From the above lemma, we can deduce the following Corollary.
\begin{corollary}
\label{cor:Xi.rhomb}
    For any $\mXi_{\setJ}$ where $\setJ\subset\{1,\dots,n+1\}$ and $2\le \#\setJ \le n$,
    $
        \rhomb{\setJ}
        $
        is a hyper-rhombus.
\end{corollary}

Next we show that $\vorzero$ is the disjoint union of
$(n+1)$ congruent hyper-rhombi.
\begin{figure*}[ht]
\def\scale{.7}
\begin{center}
\subfigure[]{
    \label{fig:A3.V.rhomb.split}
\includegraphics[height=.27\textwidth]{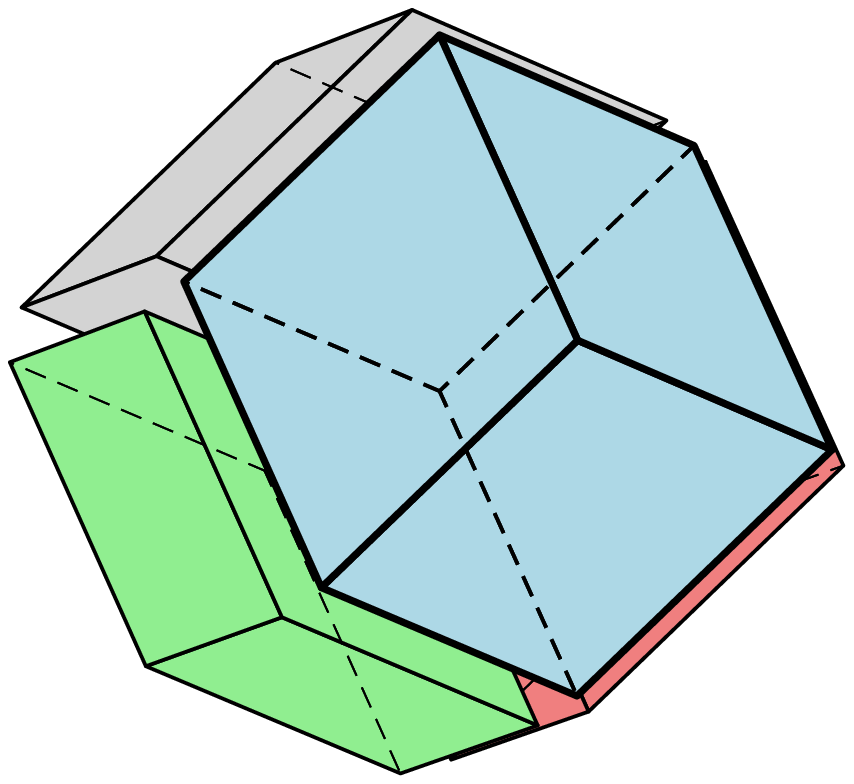} 
}
\subfigure[]{
    \label{fig:A3.V.rhomb}
\includegraphics[height=.27\textwidth]{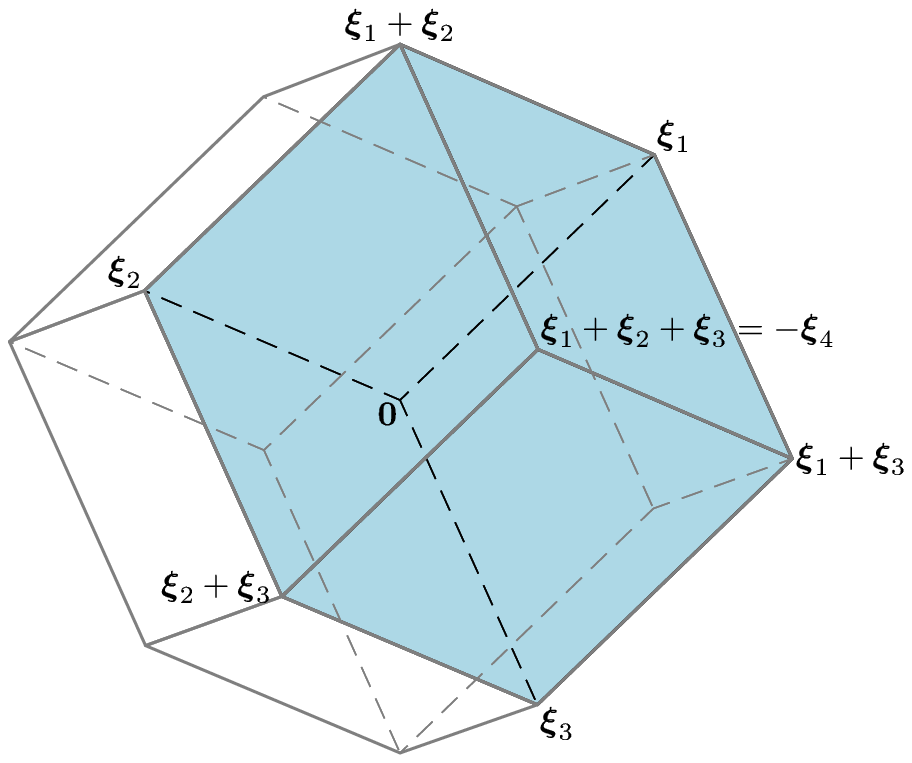} 
}
\\
\subfigure[]{
    \label{fig:A3.V.tets}
\includegraphics[height=.27\textwidth]{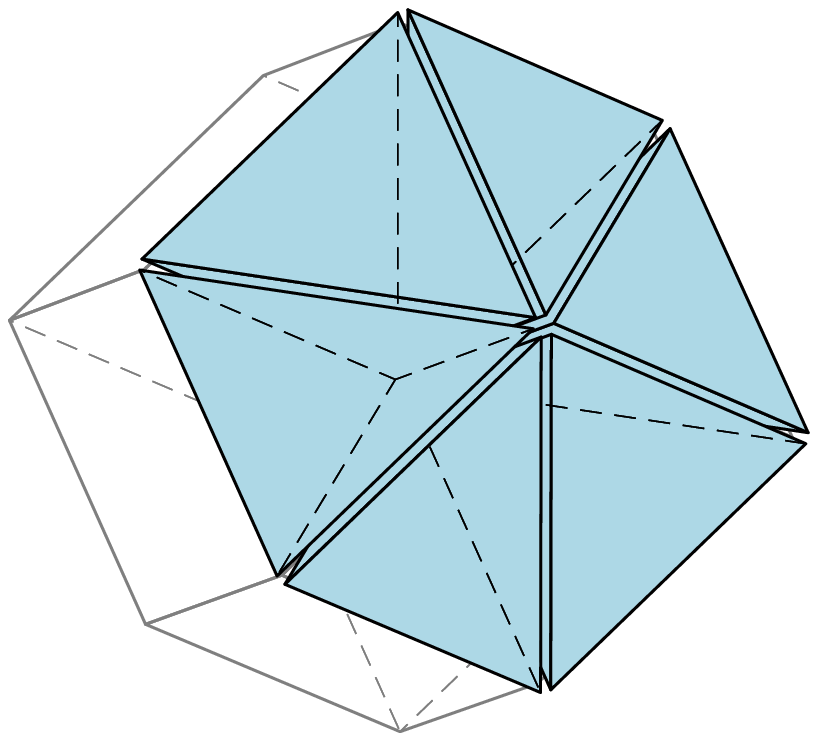} 
}
\subfigure[]{
    \label{fig:A3.V.rhomb.split.2}
\includegraphics[height=.27\textwidth]{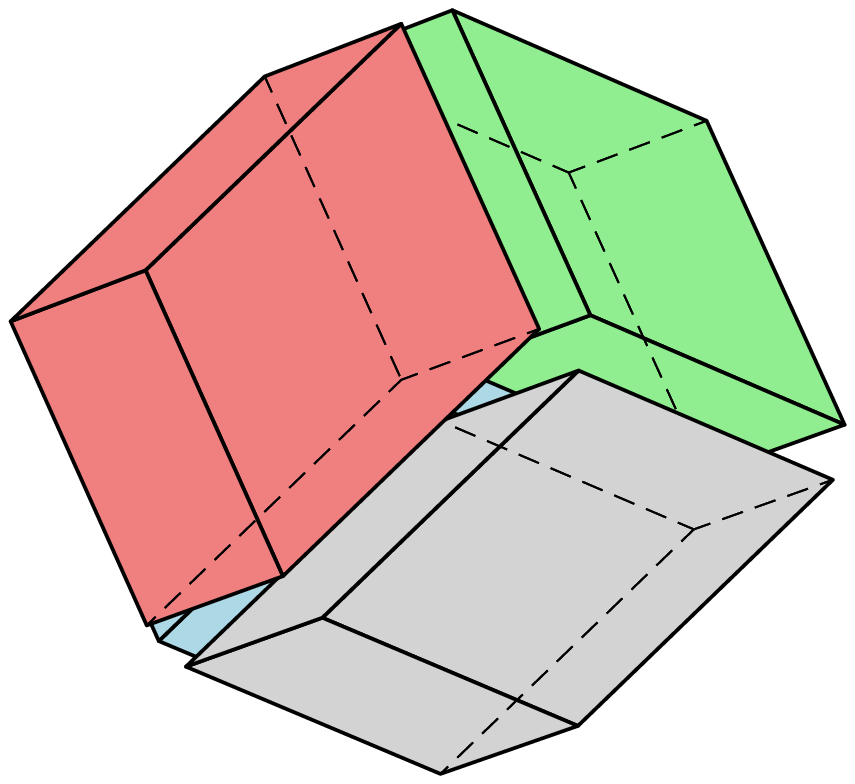} 
}

\end{center}
\caption[]{
        \subref{fig:A3.V.rhomb.split} Decomposition
        of $\vorzero$ of $\ZAn{3}$ into four hyper-rhombi:
            (green) $\rhomb{1}$,
            (red) $\rhomb{2}$,
            (gray) $\rhomb{3}$,
            and
            (blue) $\rhomb{4}$.
        \subref{fig:A3.V.rhomb} 
        $\rhomb{4}$ with its vertices.
        \subref{fig:A3.V.tets}
        Decomposition of $\rhomb{4}$
        into six tetrahedra.
    \subref{fig:A3.V.rhomb.split.2}
    The other hyper-rhombic decomposition.
}
\label{fig:A3.V.decomp}
\end{figure*}

\begin{lemma}
\label{lem:V.decomp.rhombi}
    $\vorzero$
    is the disjoint union of $(n+1)$ 
    congruent
    hyper-rhombi
    and each hyper-rhombus is further composed of $n!$ congruent simplices
    (\reffig{fig:A2.V} and \reffig{fig:A3.V.tets}).
    Specifically,
    the 
    union
    of the $n!$ images of $\simplex_0$
    \eqref{eq:simplex.0}
    under the permutations of the first $n$ components
    form
    the $n$-dimensional hyper-rhombus
    $\rhomb{n+1}$ 
    (the blue rhombus in \reffig{fig:A2.V} and
    the blue rhombohedron in \reffig{fig:A3.V.rhomb}).
    Then, 
    the 
    union of
    the $(n+1)$ images of 
    $\rhomb{n+1}$
    under the circular shifts
    in $\Gshift{n+1}$
    forms 
    $\vorzero$.
    In other words,
    \[
        \vorzero
        =
        \bigsqcup_{\mathclap{\mGamma\in\Gshift{n+1}}}\mGamma\rhomb{n+1}
        =
        \bigsqcup_{j=1}^{n+1}   
        \rhomb{j}
        \text{ where }
        \rhomb{j}=
            \Mshift{n+1}{j\bmod (n+1)}
        \rhomb{n+1}
    \]
    and
    \[
        \rhomb{n+1}
        =\bigsqcup_{\mathclap{\mQ\in \Gi}}\mQ\simplex_0
    \text{ where }
    \Gi:=
    \left\{
        \begin{bmatrix}
            \Mperm  \\
               &   1   \\
        \end{bmatrix}:
            \Mperm\in\Gsym{n}       
    \right\}.
\]

\end{lemma}
\begin{proof}
We already know that $\rhomb{n+1}$ is
a hyper-rhombus due to \reflem{lem:rhombi}.
Since
the vertices of $\setV_0$
are taken from 
\[
    \mXi_{n+1}=
    \begin{bmatrix}
        \vxi_1   &   \cdots  &   \vxi_n
    \end{bmatrix}
    =
    \begin{bmatrix}
    \matid{n} +
    \left(
        -\frac{1}{n+1}
    \right)
    \matone{n}  \\
        \trans{\vone}
    \end{bmatrix}
\]
and
due to \refcor{cor:kuhn},
the union of the images of $\simplex_0$
under $\Gi$ forms
an $n$-dimensional hyper-rhombus embedded
in the hyperplane $\vx\cdot\vunit{n+1}=1$.
Then,
the $(n+1)$ parallelepipeds
$\left\{\rhomb{j}:1\le j\le n+1\right\}$
are closed under the actions
in $\Gshift{n+1}$ since
(with $\mXi_{n+1}$ and $\mXi_j$ viewed as sets)
\[
       \Mshift{n+1}{j\bmod (n+1)}
    \mXi_{n+1}
    =
    \mXi_{j},
        \quad
    1\le j\le n+1.
\]
Finally, the union of them forms 
    $\vorzero$
due to 
\reflem{lem:V.decomp.simplex}
and the decomposition
\[
    \Gsym{n+1}=\Gshift{n+1}\Gi.
\]
\end{proof}

Note that there is another hyper-rhombic decomposition 
\begin{equation}
\label{eq:V.decomp.rhomb.2}
    \vorzero
    =
    \bigsqcup_{j=1}^{n+1}   
    \left(\vxi_j+\rhomb{j}\right)
\end{equation}
which can be obtained by `flipping' all the directions in $\mXi$
    (\reffig{fig:A3.V.rhomb.split.2}).

\begin{corollary}
We can compute the volume of the Voronoi cell as
\[
    \vol\left(\vorzero\right)
    =
    (n+1)\vol\left(\rhomb{n+1}\right)
    =
    (n+1)\sqrt{\trans{\mXi_{n+1}}\mXi_{n+1}}
    =
    \sqrt{n+1}
\]
which was verified by other researchers
\cite{conway82voronoi,koca18explicit}.
\end{corollary}

\begin{lemma}
All the $k$-dimensional faces of $\vorzero$, $2\le k\le n-1$,
are hyper-rhombi.
See \reffig{fig:A4.V.face} for $n=4$.
\end{lemma}
\begin{proof}
    The claim holds due to 
\refcor{cor:Xi.rhomb}
and
\reflem{lem:V.decomp.rhombi}.
\end{proof}
Note that \citet{koca18explicit}
proved the case of $k=n-1$.

\begin{figure*}[ht]
\begin{center}
\subfigure[]{
    \label{fig:A3.V.roof}
        \includegraphics[height=.26\textwidth]{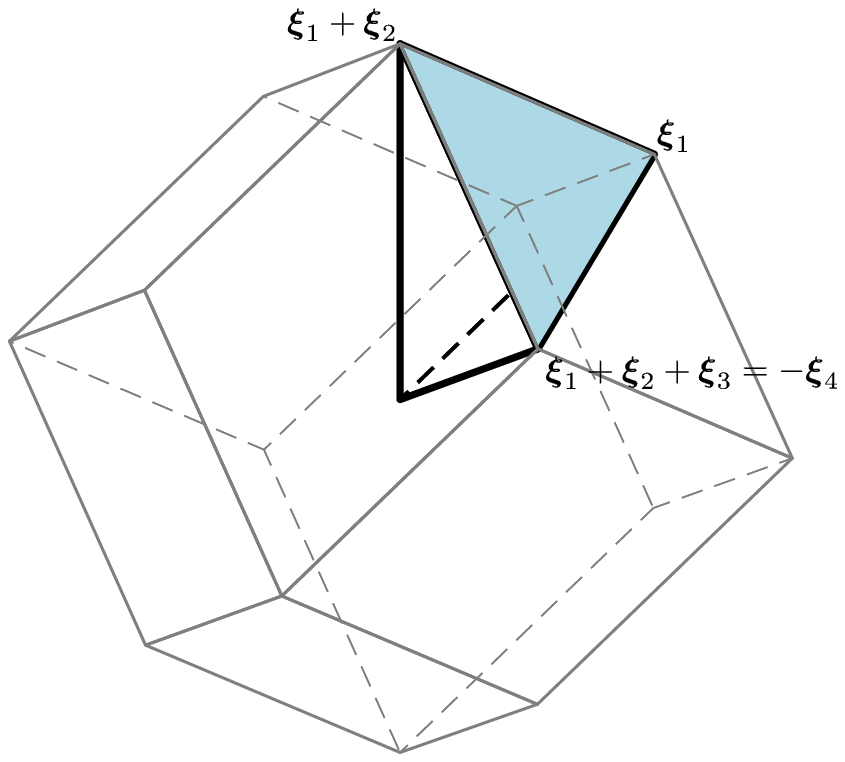}
}
\subfigure[]{
    \label{fig:A3.V.face}
        \includegraphics[height=.26\textwidth]{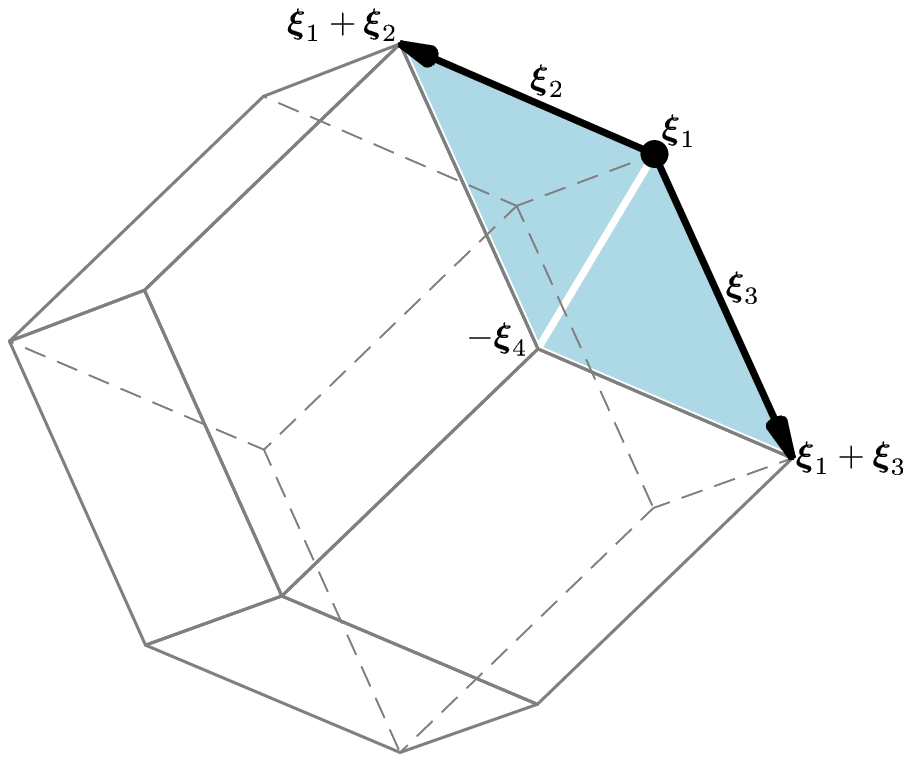}
}
\subfigure[]{
    \label{fig:A3.V.faces}
        \includegraphics[height=.26\textwidth]{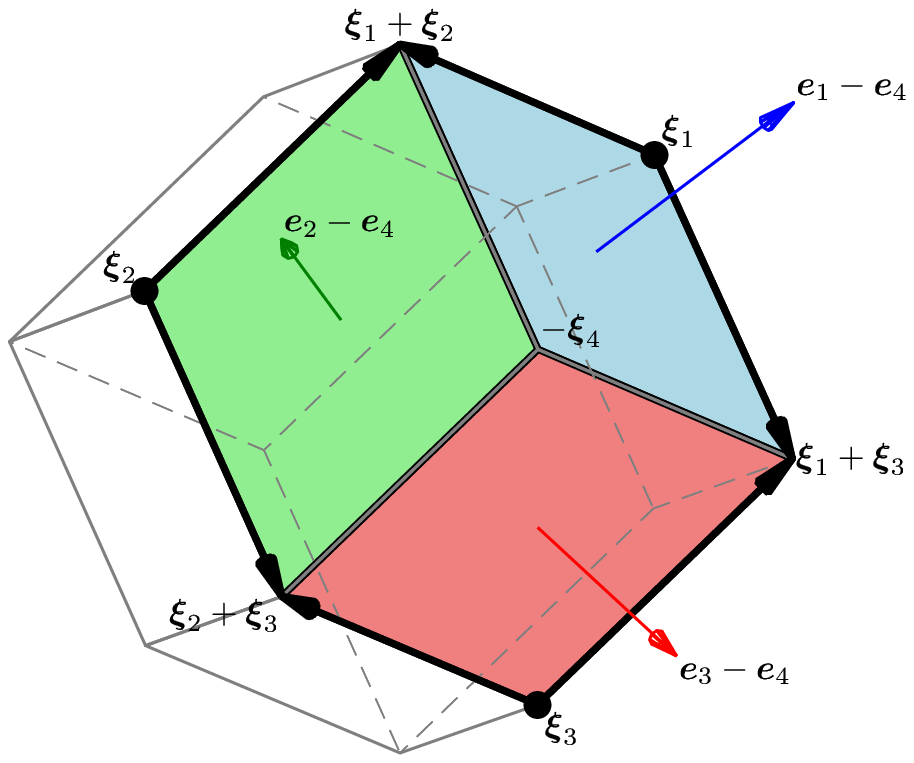}
}
\end{center}
\caption[]{
    \subref{fig:A3.V.roof} (Blue triangle) the roof of the fundamental simplex.
    \subref{fig:A3.V.face} (Blue rhombus) the hyper-rhombic face $\vxi_1+\rhomb{14}$.
    The triangle $\CH{\vxi_1,\vxi_1+\vxi_3,-\vxi_4}$
    is obtained by applying
    $
        \begin{bsmallmatrix}
            1   \\
             &  &   1   \\
             &  1   \\
             &  &   &   1
        \end{bsmallmatrix}
    $ to \subref{fig:A3.V.roof}.
    \subref{fig:A3.V.faces} Three hyper-rhombic faces
    sharing $-\vxi_4$:
    (blue) 
    $\vxi_1+\rhomb{14}$,
    (green) $\vxi_2+\rhomb{24}$,
    and
    (red) $\vxi_3+\rhomb{34}$.
    The red and green hyper-rhombic faces
    are
    obtained by applying 
    $
        \begin{bsmallmatrix}
            &   1   \\
            &   &   1   \\
            1   \\
            &   &   &   1   
        \end{bsmallmatrix}
    $
    and
    $
        \begin{bsmallmatrix}
            &   &   1   \\
            1   \\
            &   1   \\
            &   &   &   1
        \end{bsmallmatrix}
    $
    to the blue face, respectively.
    The colored arrows denote the nearest lattice points
    of $\vorzero$ associated with each face.
    The whole $12$ faces are obtained by applying
    the transformations in $\Gshift{4}$ to \subref{fig:A3.V.faces}.
}
\label{fig:V3.V.boundary}
\end{figure*}

\begin{figure*}[ht]
\begin{center}
    \includegraphics[width=.4\textwidth]{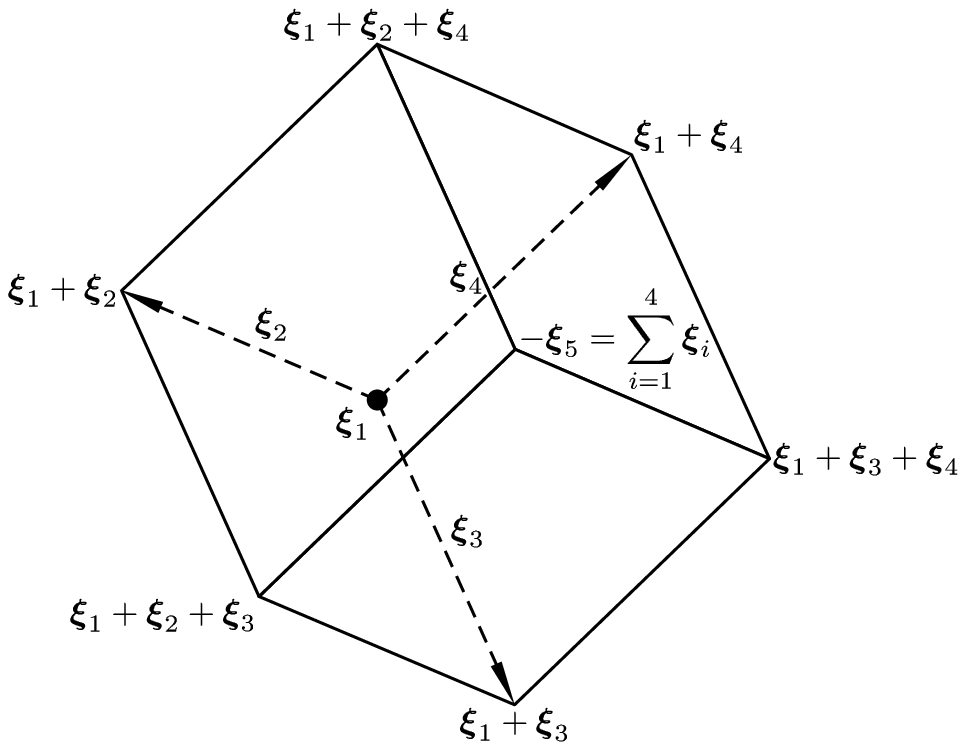}
\end{center}
\caption{
    The hyper-rhombic face $\vxi_1+\rhomb{15}$
    of 
    $\vorzero$ of $\ZAn{4}$.
The tetrahedron 
$\CH{\left\{\sum_{i=1}^j\vxi_i:1\le j\le 4\right\}}$
is the roof of $\simplex_0$.
}
\label{fig:A4.V.face}
\end{figure*}

Next, we analyze the structure of
the \emph{cells}, $(n-1)$-dimensional faces, of $\vorzero$.
We first determine the explicit structure
of the hyper-rhombic face
that contains the roof of the fundamental simplex.
\begin{lemma}
\label{lem:rhomb.face}
    The hyper-rhombic cell 
    of $\vorzero$ that contains the roof of $\simplex_0$
    is 
    \[
        \vxi_1+\rhomb{1,n+1}
    \]
    which is composed of the $(n-1)!$ images 
    of the roof 
    of $\simplex_0$
    under the group
    \begin{equation}
    \label{eq:G2}
        \Gii:=
        \left\{
        \begin{bsmallmatrix}
            1   \\
            &   \Mperm \\
            &   &   1
        \end{bsmallmatrix}:
                \Mperm\in\Gsym{n-1}
        \right\}
    \end{equation}
    and they
    all share the edge connecting $\vxi_1$ and $-\vxi_{n+1}$.
    Refer to \reffig{fig:A3.V.face} and \reffig{fig:A4.V.face}
    for $n=3$ and $n=4$, respectively.
\end{lemma}
\begin{proof}
Due to \reflem{lem:V.decomp.simplex},
    the roof of $\simplex_0$ is
    \[
        \CH{
        \left\{
            \vv_j
            =\vxi_1+
            \sum_{i=2}^j
            :1\le j\le n
        \right\}
        }
    \]
    where the vertices are taken from
    the $(n+1)\times (n-1)$ matrix
    \[
        \mXi_{1,n+1}=
        \begin{bmatrix}
            \\
            \vxi_2   &   \cdots  &   \vxi_n   \\
            \\
        \end{bmatrix}
        =
        \begin{bmatrix}
            \trans{\vone}   \\
            \matid{n-1}
            +\left(-\frac{1}{n+1}\right)\matone{n-1} \\
            \trans{\vone}   \\
        \end{bmatrix}.
    \]
     Due to \refcor{cor:kuhn},
    the union of the images of $\CH{\left\{\vv_j:2\le j\le n\right\}}$
    under $\Gii$ forms the $(n-1)$-dimensional hyper-rhombus
    $\rhomb{1,n+1}$ embedded in the $(n-1)$-dimensional hyperplane
    \[
        \left\{
            \vx\cdot\vunit{1}=1:\vx\in\R^{n+1}
        \right\}
        \cap
        \left\{
            \vx\cdot\vunit{n+1}=1:\vx\in\R^{n+1}
        \right\}
    \]
    and they all share the edge connecting  $\vzero$ and $\sum_{i=2}^n\vxi_i$.
    Moreover, 
    since $\vxi_1$ is invariant under $\Gii$
    the `shifted' simplices all share the edge connecting
    $\vxi_1$ and $\vxi_1+\sum_{i=2}^n\vxi_i=-\vxi_{n+1}$.
\end{proof}

Next we show that the vertex $-\vxi_{n+1}$ 
of $\vorzero$
is surrounded by $n$ non-overlapping congruent hyper-rhombic cells.
\begin{lemma}
\label{lem:rhomb.vert}
On the boundary of $\vorzero$,
    there are $n$ congruent hyper-rhombic cells sharing the vertex $-\vxi_{n+1}$.
    These cells are 
    \begin{enumerate}[(i)]
       \item     the images of
    $\vxi_1+\rhomb{1,n+1}$ under the  
    group of the circular shifts of the first $n$ components and 
    \item   the $(n-1)$-dimensional
    hyper-rhombic cells of $\rhomb{n+1}$ 
    adjacent to $-\vxi_{n+1}$.
    \end{enumerate}
    Refer to \reffig{fig:A2.V} and
    \reffig{fig:A3.V.faces}.

\end{lemma}
\begin{proof}
    Let
    \begin{equation}
    \label{eq:G3}
        \Giii:=
        \left\{
                \begin{bmatrix}
                    \Mshift{n}{i} \\
                    &   1
                \end{bmatrix}:
                \Mshift{n}{i}\in\Gshift{n},
                    0\le i\le n-1
        \right\}.
    \end{equation}
    Since 
        $\Gii,\Giii\le \Gsym{n+1}$,
        $\Gii\cap \Giii=\{\matid{n+1}\}$,
        and all the cells of $\vorzero$ are composed of the images
        of the roof of $\simplex_0$ under $\Gsym{n+1}$,
        the $n$ images of $\vxi_1+\rhomb{1,n+1}$ under $\Giii$
        do not overlap.

Moreover, since ($0\le i\le n$)
    \[
        \begin{bmatrix}
            \Mshift{n}{i}   \\
            &   1
        \end{bmatrix}
    (\vxi_1+\rhomb{1,n+1})=\vxi_{i+1}+\rhomb{i+1,n+1}
    \text{ and }
    \vxi_{i+1}+\sum_{
            \mathclap{
            \substack{k=1\\k\ne i+1,n+1}}}^{n+1}\vxi_k
    =\sum_{k=1}^n\vxi_k
    =-\vxi_{n+1},
    \]
    the $n$ images of $\vxi_1+\rhomb{1,n+1}$
    under $\Giii$ share the vertex $-\vxi_{n+1}$.
\end{proof}

Finally, the next lemma shows that the boundary of $\vorzero$ is composed
of  $n(n+1)$ congruent hyper-rhombic cells
associated with the nearest neighbor lattice points of $\An{n}$
and determines their explicit structures.

\begin{lemma}
    The $(n+1)$ images  of $n$ hyper-rhombic cells sharing $-\vxi_{n+1}$
(\reflem{lem:rhomb.vert})
    under 
    $\Gshift{n+1}$
    form the boundary of $\vorzero$.
    Overall, the boundary of $\vorzero$ is composed of $n(n+1)$ congruent hyper-rhombic faces
    \[
        \{
            \vxi_i+\rhomb{ij}:
            1\le i\ne j\le n+1
        \}.
    \]
    Below are the properties of the face $\vxi_i+\rhomb{ij}$.
    (In the following, we denote $\Mshift{n}{i\bmod n}=\Mshift{n}{i}$ for simplicity.)
    \begin{enumerate}[(1)]
    \item   
    It is obtained 
    by transforming  $\vxi_1+\rhomb{1,n+1}$
    with
        \[
            \begin{bmatrix}
            \Mshift{n+1}{j}
            \end{bmatrix}
            \begin{bmatrix}
                \Mshift{n}{i-j-1}  \\
                    &   1
            \end{bmatrix}.
        \]
        \item
        It is embedded in the bisecting hyperplane of the nearest neighbor point
            $\vunit{i}-\vunit{j}\in\ZAn{n}$.
            \item
            It is composed of the images of the roof of $\simplex_0$
            that share the edge connecting $\vxi_i$ and $-\vxi_j$.
            \end{enumerate}

\end{lemma}
\begin{proof}
    Since 
    \eqref{eq:G2}
    \eqref{eq:G3}
    \[
        \Gshift{n+1},\Gii,\Giii\le\Gsym{n+1}
        \text{ and }
        \Gshift{n+1}\cap \Gii\cap \Giii=\{\matid{n+1}\},
        \]
        the $n(n+1)$ images of $\vxi_1+\rhomb{1,n+1}$ under $\Gshift{n+1}\Giii$
        do not overlap.
            Moreover,
            since $|\Gshift{n+1}\Giii\Gii|=|\Gsym{n+1}|$,
            $\Gsym{n+1}$ can be
            represented as
            \[
                \Gsym{n+1}=\Gshift{n+1}\Giii \Gii  
\]
            and therefore the $n(n+1)$ images of $\vxi_1+\rhomb{1,n+1}$
            under $\Gshift{n+1}\Giii$
            form the boundary of $\vorzero$.

\begin{enumerate}[(1)]
\item
    The claim holds since
    \[
        \Mshift{n+1}{j}
        \begin{bmatrix}
            \Mshift{n}{i-j-1}  \\
                &   1
        \end{bmatrix}
        (\vxi_1+\rhomb{1,n+1})
        =
        \Mshift{n+1}{j}
        (\vxi_{i-j}+\rhomb{i-j,n+1})
        =
            \vxi_{i}
            +\rhomb{i,j}.
        \]

\item
    The claim holds since
    $
        \vxi\cdot(\vunit{i}-\vunit{j})=0$
         for 
             $
        \vxi\in\mXi_{ij}$.
\item
    Due to
\reflem{lem:rhomb.face},
    $\vxi_i+\rhomb{ij}$
    is composed of the $(n-1)!$ 
    simplices sharing the edge connecting 
    \[
    \vxi_i
        \text{  and }
    \vxi_i+\sum_{
            \mathclap{
                \substack{
                    k=1\\k\ne i,j}}}^{n+1}\vxi_k=-\vxi_j.
    \]
\end{enumerate}

\end{proof}

\section{Voronoi Cell as the Vertex-First Projection of the Unit Cube}
\begin{figure*}
\begin{center}
    \includegraphics[width=.4\textwidth]{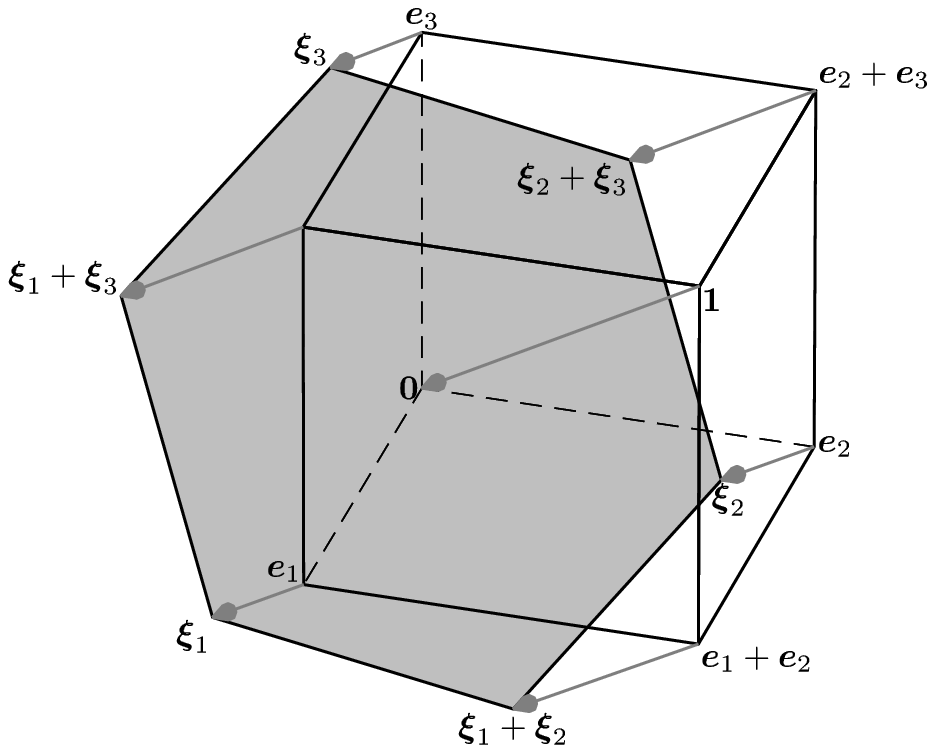}
\end{center}
\caption[]{
    The orthogonal
        projection of $\cube{3}{}$ onto $\plnone$ along $\vone$
        is the Voronoi cell of $\An{2}$.
}
\label{fig:proj}
\end{figure*}

We show that $\vorzero$
is the orthogonal projection
of the $(n+1)$-dimensional unit cube onto $\plnone$
along the `diagonal direction' $\vone$,
called the \emph{vertex-first projection}
\cite{coxeter73regular}.
\begin{lemma}
    $\vorzero$ is the projection
of 
$\cube{n+1}{}$ via $\mXi$.
See \reffig{fig:proj} for $n=2$.
\end{lemma}
\begin{proof}
Note that
the $(n+1)$-simplex $\simplexCube_0$ 
\eqref{eq:kuhn.simplex}
    is projected to
    the 
    fundamental simplex 
    $\simplex_0$
    \eqref{eq:verts.simplex.0}
    since 
    \[
        \mXi\vu_j=
        \begin{cases}
            \vzero=\vv_0  &   j\in\{0,n+1\}   \\
            \sum_{i=1}^j\vxi_i=\vv_j  &   1\le j\le n.
        \end{cases}
    \]
    Specifically, the $n$ vertices 
    $\left\{\vv_j
        :1\le j\le n\right\}$ form the roof
    of $\simplex_0$
(\reflem{lem:V.decomp.simplex}).
    Since 
    \[
        \mXi\Mperm=\Mperm\mXi,
        \quad \Mperm\in\Gsym{n+1}      
    \]
    and
    $\vorzero$ is the union of
    the $(n+1)!$ images of $\simplex_0$
    under $\Gsym{n+1}$, the claim holds.
\end{proof}
We can deduce that
the projections of
the vertices of $\cube{n+1}{}$, excluding $\vzero$ and $\vone$,
are the vertices of $\vorzero$
and
therefore $\vorzero$ is the convex hull of these $(2^{n+1}-2)$ vertices.

Alternatively,
$\vorzero$
is a \emph{zonotope}
generated by the line segments $\vxi_1,\dots,\vxi_{n+1}.$
\cite{ziegler95lectures}:
\[
    \vorzero
    =
    \mXi\cube{n+1}{}
    =
    \left\{
        \sum_{j=1}^{n+1}
        t_j\vxi_j:0\le t_j\le 1
    \right\}.
\]

Next we show that the $(n+1)$ composing hyper-rhombi
of $\vorzero$
are the projections of the  cubic cells of $\cube{n+1}{}$
adjacent to the vertex $\vzero$ via $\mXi$
(\reffig{fig:proj}).

\begin{lemma}
The composing hyper-rhombi 
of $\vorzero$
(\reflem{lem:V.decomp.rhombi})
are the orthogonal projections of the $n$-dimensional cubic cells
of $\cube{n+1}{}$
adjacent to $\vzero$
via $\mXi$.
In other words,
\[
    \rhomb{j}=\mXi\cube{n+1}{j},
    \quad
    1\le j\le n+1.
\]
\end{lemma}
\begin{proof}
The claim holds since 
\begin{align*}
    \rhomb{j}
    &=
    \left\{\sum_{\mathclap{\substack{i=1\\i\ne j}}}^{n+1}t_i\vxi_i:0\le t_i\le 1\right\}
    =
    \left\{\sum_{\mathclap{\substack{i=1\\i\ne j}}}^{n+1}t_i\mXi\vunit{i}:0\le t_i\le 1\right\}
    \\
    &=
    \left\{\mXi\sum_{\mathclap{\substack{i=1\\i\ne j}}}^{n+1}t_i\vunit{i}:0\le t_i\le 1\right\}
    =
    \mXi\cube{n+1}{j}.
\end{align*}
\end{proof}

Note that the other hyper-rhombic decomposition
in
\eqref{eq:V.decomp.rhomb.2}
are the projection of the faces of $\cube{n+1}{}$
via $\mXi$
adjacent to the vertex $\vone$
(\reffig{fig:proj}).
In other words,
\[
    \vxi_j+\rhomb{j}
    =
    \mXi
    \left(
        \vunit{j}+\cube{n+1}{j}
        \right),
        \quad
        1\le j\le n+1.
\]

\section{Voronoi Cell 
as the Section of the Voronoi Cell of the $\Dn{n+1}$ Lattice}

\begin{figure*}
\begin{center}
\includegraphics[width=.4\textwidth]{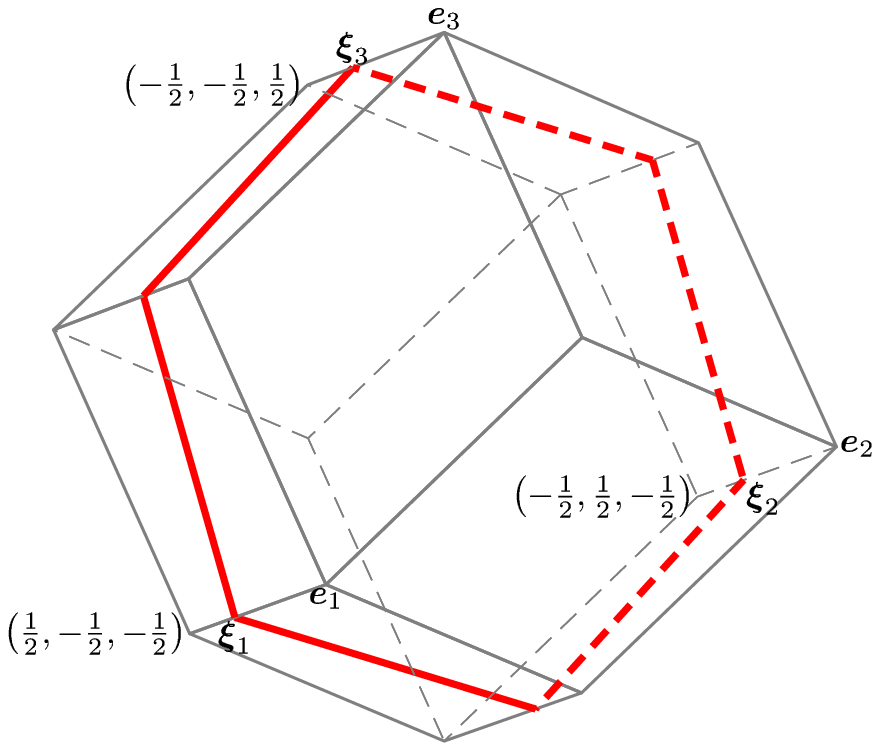}
\end{center}
\caption[]{
    The section of the Voronoi cell of the $\Dn{3}$ lattice with $\plnone$
    is the Voronoi cell of the $\An{2}$ lattice.
}
\label{fig:section}
\end{figure*}
The following lemma
shows that
$\vorzero$ is the 
intersection
of the Voronoi cell of the $\Dn{n+1}$ lattice,
called 
\emph{section} \cite{coxeter73regular},
with $\plnone$ in low dimensions.
\begin{lemma}
Let
\[
    \ZDn{n+1}:=
    \left\{
        \vx\in\Z^{n+1}:\vx\cdot\vone\text{ is even}
    \right\}
\]
be the $\Dn{n+1}$ lattice
and $\vor{\Dn{n+1}}{\vp}$ be the Voronoi cell
at $\vp\in\ZDn{n+1}$.
In dimensions $n\in\{1,2,3\}$ only,
the following holds.
\[
    \vor{}{\vp}=
    \vor{\Dn{n+1}}{\vp}\cap\plnone,
    \quad
    \vp\in\ZAn{n}.
\]
\end{lemma}
\begin{proof}
Since
\[
    \ZAn{n}=
    \Z^{n+1}\cap\plnone
    =   
    \ZDn{n+1}\cap\plnone,
\]
we only need to show that 
\[
    \dim\left(\plnone \cap \vor{\Dn{n+1}}{\vp}\right)<n,\quad \forall
    \vp\notin\ZAn{n}.
\]

Note that the closest lattice points of $\ZDn{n+1}$ to $\plnone$
which are not on $\plnone$ are those of the form
\[
    \vp\pm\left(\vunit{i}+\vunit{j}\right),
    \quad
    \vp\in\ZAn{n}
    \text{ and }
    1\le i\ne j\le n+1.
\]
Due to the symmetry, we only need to analyze the intersection
of $\vor{\Dn{n+1}}{\vunit{1}+\vunit{2}}$ with $\plnone$.

It is known that $\vor{\Dn{n+1}}{\vp}$ 
is the 
    \emph{pyramidal cube} \cite{conway91cell}
whose vertices are
\[
    \left\{
            \vp+
    \left(
        \pm\frac{1}{2},\cdots,\pm\frac{1}{2}
    \right)
    \right\}
    \cup
    \left\{
            \vp
           \pm\vunit{j}:
           1\le j\le n+1 
    \right\}
\]
and therefore
for $\vp=\vunit{1}+\vunit{2}$,
\[
    -\frac{n-3}{2}
    \le
    \left(
        (\vunit{1}+\vunit{2})
        +
    \left(
        \pm\frac{1}{2},
        \cdots,\pm\frac{1}{2}
    \right)
    \right)
    \cdot\vone
    \le
    \frac{n+5}{2}
\]
and
\[
    \left(
        (\vunit{1}+\vunit{2})
        \pm
        \vunit{j}
    \right)
    \cdot\vone
    \ge 1,
    \quad
    1\le j\le n+1.
\]
Therefore,
for $n\in\{1,2\}$,
all the vertices of $\vor{\Dn{n+1}}{\vunit{1}+\vunit{2}}$ belong to the half space $\{\vx\cdot\vone>0:\vx\in\R^{n+1}\}$.
We only need to show that 
$\vor{\Dn{4}}{(1,1,0,0)}$
trivially intersect $\plnone$,
which can be easily verified since the only vertex of
$\vor{\Dn{4}}{(1,1,0,0)}$
that intersects $\plnone$ is
$
    \left(
    \frac{1}{2},
    \frac{1}{2},
    -\frac{1}{2},
    -\frac{1}{2}\right)
    $.
For $n>3$,
the vertex
$
    (
    \frac{1}{2},
    \frac{1}{2},
    -\frac{1}{2},
    \cdots,
    -\frac{1}{2},
    )\cdot\vone<0
$
and therefore $\vor{\Dn{n+1}}{\vunit{1}+\vunit{2}}$
intersect $\plnone$ with $n$-dimensional volume.

\end{proof}
Specifically, the Voronoi cell of $\ZDn{3}$ 
is
the rhombic dodecahedron
whose section with $\plnone$ is the regular hexagon
(\reffig{fig:section})
\[
    \CH{
            \left\{
         \pm\frac{2}{3}   
        \perm\left( 2,-1,-1 \right)
            \right\}
    }
\]
where $\perm$ denotes all the permutations
of the coordinate values.
The Voronoi cell of $\ZDn{4}$
is the \emph{$24$-cell}
whose section with $\plnone$ is the rhombic
dodecahedron
\[
    \CH{
    \left\{
            \frac{1}{2}
            \perm\left(
        1,1,-1,-1
        \right)
            \right\}
    \cup
    \left\{
            \pm\frac{1}{4}
            \perm(3,-1,-1,-1)
        \right\}
    },
\]
which is transformed to 
\[
    \CH{
            \left\{
            \perm(\pm1,0,0)
            \right\}
            \cup
            \left\{
            \left(\pm\frac{1}{2},\pm\frac{1}{2},\pm\frac{1}{2}\right)
            \right\}
    }
\]
by $\Mproj{3}$.

\section{Conclusion}
We characterized the combinatorial structure of the Voronoi cell
of the $\An{n}$ lattice.
Based on the well-known fact that it is composed of $(n+1)!$ congruent
simplices, we showed that it is composed of $(n+1)$ congruent
hyper-rhombi. 
We analyzed the explicit structure of the hyper-rhombic cells
of the Voronoi cell, including the fact that all the $k$-dimensional faces,
$2\le k\le n-2$, are hyper-rhombic.
We showed that it is the vertex-first projection
of the $(n+1)$-dimensional unit cube and finally
verified that in low dimensions, $n\le 3$, it is the section
of the Voronoi cell of the $\Dn{n+1}$ lattice with the diagonal
hyperplane.
All the analyses are done algebraically without any Coxeter-Dynkin diagrams
so they are more accessible to those who are not familiar with them.

\section*{Acknowledgments}
        This work was supported by the National Research Foundation of Korea(NRF) grant 
        funded by the Korea government(MSIT) (No.  2021R1F1A1060215).

\bibliographystyle{plainnat}
\bibliography{p}

\begin{thebibliography}{10}
\providecommand{\natexlab}[1]{#1}
\providecommand{\url}[1]{\texttt{#1}}
\expandafter\ifx\csname urlstyle\endcsname\relax
  \providecommand{\doi}[1]{doi: #1}\else
  \providecommand{\doi}{doi: \begingroup \urlstyle{rm}\Url}\fi

\bibitem[Conway and Sloane(1982)]{conway82voronoi}
John~Horton Conway and Neil James~Alexander Sloane.
\newblock Voronoi regions of lattices, second moments of polytopes, and
  quantization.
\newblock \emph{IEEE Transactions on Information Theory}, 28\penalty0
  (2):\penalty0 211--226, March 1982.
\newblock \doi{10.1109/TIT.1982.1056483}.

\bibitem[Conway and Sloane(1991)]{conway91cell}
John~Horton Conway and Neil James~Alexander Sloane.
\newblock The cell structures of certain lattices.
\newblock In P.~Hilton, F.~Hirzebruch, and R.~Remmert, editors,
  \emph{Miscellanea Mathematica}, page 71–107. Springer Berlin Heidelberg,
  1991.
\newblock \doi{10.1007/978-3-642-76709-8_5}.

\bibitem[Conway and Sloane(1998)]{conway98sphere}
John~Horton Conway and Neil James~Alexander Sloane.
\newblock \emph{Sphere packings, lattices and groups}, volume 290.
\newblock Springer Science \& Business Media, 1998.
\newblock \doi{10.1007/978-1-4757-6568-7}.

\bibitem[Coxeter(1934)]{coxeter34discrete}
Harold Scott~MacDonald Coxeter.
\newblock Discrete groups generated by reflections.
\newblock \emph{The Annals of Mathematics}, 35\penalty0 (3):\penalty0 588--621,
  July 1934.
\newblock \doi{10.2307/1968753}.

\bibitem[Coxeter(1973)]{coxeter73regular}
Harold Scott~MacDonald Coxeter.
\newblock \emph{Regular Polytopes}.
\newblock Dover Publications, 3rd edition, 1973.
\newblock ISBN 0-486-61480-8.

\bibitem[Kim and Peters(2011)]{kim11symmetric}
Minho Kim and J{\"o}rg Peters.
\newblock Symmetric box-splines on root lattices.
\newblock \emph{Journal of Computational and Applied Mathematics}, 235\penalty0
  (14):\penalty0 3972--3989, May 2011.
\newblock \doi{10.1016/j.cam.2010.11.027}.

\bibitem[Koca et~al.(2018)Koca, Koca, Al-Siyabi, and Koc]{koca18explicit}
Mehmet Koca, Nazife~Ozdes Koca, Abeer Al-Siyabi, and Ramazan Koc.
\newblock Explicit construction of the {Voronoi} and {Delaunay} cells of
  {$W(A_n)$} and {$W(D_n)$} lattices and their facets.
\newblock \emph{Acta Crystallographica Section A Foundations and Advances},
  74\penalty0 (5):\penalty0 499–511, September 2018.
\newblock \doi{10.1107/S2053273318007842}.

\bibitem[Kuhn(1960)]{kuhn60some}
Harold~William Kuhn.
\newblock Some combinatorial lemmas in topology.
\newblock \emph{IBM Journal of Research and Development}, 4\penalty0
  (5):\penalty0 518–524, November 1960.
\newblock \doi{10.1147/rd.45.0518}.

\bibitem[Moody and Patera(1992)]{moody92voronoi}
Robert~Vaughan Moody and Ji\v{r}\'i Patera.
\newblock {Voronoi} and {Delaunay} cells of root lattices: classification of
  their faces and facets by {Coxeter}-{Dynkin} diagrams.
\newblock \emph{Journal of Physics A: Mathematical and General}, 25\penalty0
  (19):\penalty0 5089--5134, October 1992.
\newblock \doi{10.1088/0305-4470/25/19/020}.

\bibitem[Ziegler(1995)]{ziegler95lectures}
G\"unter~Matthias Ziegler.
\newblock \emph{Lectures on Polytopes}, volume 152 of \emph{Graduate Texts in
  Mathematics}.
\newblock Springer-Verlag, 1995.

\end{thebibliography}

\end{document}